\documentclass[a4paper, 12pt, twoside, reqno, openright]{amsart}
\usepackage{amsmath}
\usepackage{amssymb}
\usepackage{amsthm}
\usepackage{amscd}
\usepackage{amsopn}
\usepackage[german, french, UKenglish]{babel}
\usepackage{pifont}
\usepackage{lmodern}
\usepackage[cal=boondox, bb=ams, scr=boondox]{mathalfa}
\usepackage[top=1.5in, bottom=1.2in, left=1.2in, right=1.2in]{geometry}
\usepackage[all,cmtip]{xy}

\newtheoremstyle{Teorema}{5pt}{5pt}{\it}{}{\bf}{.}{ }{}
\theoremstyle{Teorema}
\newtheorem{Theorem}{Theorem}

\newtheorem{Corollary}[Theorem]{Corollary}
\newtheorem{Proposition}[Theorem]{Proposition}

\newtheoremstyle{Annotazione}{5pt}{5pt}{\rm}{}{\bf}{.}{ }{}
\theoremstyle{Annotazione}

\def\Aut{\operatorname{Aut}}
\def\Out{\operatorname{Out}}
\def\Inn{\operatorname{Inn}}

\def\qbar{\overline{\bbq }}

\def\pioneet{\pi_1^{\text{\textup{\'{e}t}}}}

\def\SL{\mathrm{SL}}
\def\PSL{\mathrm{PSL}}

\def\GL{\mathrm{GL}}

\def\bbz{\mathbb{Z}}
\def\bbr{\mathbb{R}}
\def\bbq{\mathbb{Q}}
\def\bbc{\mathbb{C}}

\def\bbp{\mathbb{P}}
\def\bbf{\mathbb{F}}
\def\bbn{\mathbb{N}}

\def\Gal{\operatorname{Gal}}
\def\Jac{\operatorname{Jac}}

\def\defined{\overset{\mathrm{def}}{=}}

\makeatletter
\def\blfootnote{\xdef\@thefnmark{}\@footnotetext}
\makeatother

\begin{document}

\title{The Galois action on Hurwitz curves}
\author{Robert A. Kucharczyk}
\address{Universit\"{a}t Bonn\\ Mathematisches Institut\\ Endenicher Allee 60\\ 53115 Bonn\\ Germany}
\email{rak@math.uni-bonn.de}
\thanks{Research supported by the European Research Council}
\begin{abstract}
We make some observations concerning the Galois actions on Hurwitz curves and on the closely related but lesser-known Hurwitz origamis.
\end{abstract}
\maketitle

\thispagestyle{empty}

\tableofcontents

\noindent In this note we draw some consequences concerning Hurwitz curves of a recent result in the theory of dessins d'enfants by Gonz\'{a}lez-Diez and Jaikin-Zapirain (Theorem~\ref{GDJZ} below). The material presented here is a revision of part of Chapter~3 of the author's dissertation~\cite{MyThesis}, with significant changes in style and content.

\section{Faithful Galois actions on some classes of dessins}

\noindent Recall that a \emph{Bely\u{\i} pair} $(X,\beta )$ consists of a smooth projective curve $X$ over the algebraic closure $\qbar\subset\bbc$ of the rationals and a nonconstant rational map of algebraic curves $\beta\colon X\to\bbp^1_{\qbar }$ which is unramified outside $\{ 0,1,\infty \} \subset\bbp^1(\qbar )$. It is well-known (and a major motivation for the study of Bely\u{\i} pairs) that there is a correspondence, one-to-one up to isomorphism, between Bely\u{\i} pairs and the comparatively elementary combinatorial objects known as \emph{dessins d'enfants}. We assume these notions as known, otherwise we refer the reader to the many existing introductions including \cite{BogomolovHusemoeller2000, MR2895884, MR1483107, MR2349672, MR2516744, MR2070647, MR2074061, MR1305393}.

Since the absolute Galois group $\Gal (\qbar /\bbq )$ acts on the isomorphism classes of Bely\u{\i} pairs in an obvious way, we obtain a far-from-obvious action on the isomorphism classes of dessins d'enfants via this correspondence. The foundational result about dessins d'enfants is the following result, already implicit in \cite{MR1483107} but, to the author's knowledge, first stated explicitly in~\cite{MR1305393}:
\begin{Theorem}[A.~Grothendieck, L.~Schneps, \ldots ]\label{GalFaithfulOnAllDessins}
The action of $\Gal (\qbar /\bbq )$ on isomorphism classes of dessins d'enfants is faithful.
\end{Theorem}
This is essentially a consequence of Bely\u{\i}'s theorem \cite[Theorem~4]{MR534593} that every smooth projective curve over $\qbar$ can be completed to a Bely\u{\i} pair, and $\Gal (\qbar /\bbq )$ certainly acts faithfully on the isomorphism classes of such curves.

Theorem~\ref{GalFaithfulOnAllDessins} has prompted many people, most notably Grothendieck himself, to study dessins d'enfants and their Galois actions, mainly with the hope of obtaining new information about the notorious group $\Gal (\qbar /\bbq )$. It is perhaps too early to judge whether this hope was realistic; nonetheless, the theory of dessins provides us with many fine examples where algebraic number theory interacts with the geometry of algebraic curves.

Since the class of dessins d'enfants is rather large and poor in structure, a natural question following Theorem~\ref{GalFaithfulOnAllDessins} is whether there exist some highly restrictive classes of dessins on which $\Gal (\qbar /\bbq )$ still acts faithfully. For instance, $\Gal (\qbar /\bbq )$ acts faithfully on the set of dessins whose underlying surface is the sphere and whose graph is a tree, see \cite{MR1305393}.

To state the theorem alluded to in the beginning, let us call a dessin \emph{regular} if the corresponding Bely\u{\i} map $\beta\colon X\to\bbp^1_{\qbar }$ is a Galois covering. It is of \emph{type} $(p,q,r)$ if the ramification orders of $\beta$ at $0$, $1$ and $\infty$ divide $p$, $q$ and $r$, respectively.
\begin{Theorem}[G.~Gonz\'{a}lez-Diez, A.~Jaikin-Zapirain]\label{GDJZ}
Let $p,q,r\in\bbn$ with
\begin{equation}\label{HyperbolicInequality}
\frac 1p+\frac 1q+\frac 1r<1.
\end{equation}
Then $\Gal (\qbar /\bbq )$ acts faithfully on the set of isomorphism classes of regular dessins d'efants of type $(p,q,r)$.
\end{Theorem}
This is \cite[Theorem~29]{GonzalezDiez2013}. The inequality (\ref{HyperbolicInequality}) is easily seen to be strictly necessary. If it is not satisfied the curves $X$ obtained by such dessins are $\bbp^1_{\qbar }$ and the elliptic curves of $j$-invariants $0$ and $1728$, and the regular dessins these support are quite few.

\section{Hurwitz curves}

\noindent By a well-known theorem of Hurwitz~\cite{MR1510753} a (smooth projective) curve of genus $g\ge 2$ over $\bbc$ has no more than $84(g-1)$ automorphisms. Curves which attain this bound are called \emph{Hurwitz curves}.  The following is a simple consequence of Theorem~\ref{GDJZ}:
\begin{Corollary}\label{TheoremA}
The absolute Galois group $\operatorname{Gal}(\qbar /\bbq )$ operates faithfully on the set of isomorphism classes of Hurwitz curves.
\end{Corollary}
This is to be understood as follows: every Hurwitz curve has a unique model over $\qbar$, and conjugating it by an automorphism of $\qbar$ will yield another, possibly different, Hurwitz curve.
\begin{proof}
If $X$ is a Hurwitz curve, then $X/\Aut X$ is a curve of genus zero, and the projection map $X\to X/\Aut X$ is unramified outside of three points in the target, where it has ramification orders $2$, $3$, and $7$, respectively. Conversely, if $\beta\colon X\to\bbp^1_{\qbar }$ is a Galois Bely\u{\i} map of type $(2,3,7)$, then $X$ is a Hurwitz curve and $\beta$ is the orbit projection for $\Aut X$. For the proof of these classical facts, see \cite{MR1722414}.

Hurwitz curves are therefore in one-to-one correspondence to regular dessins of type $(2,3,7)$, and Theorem~\ref{GDJZ} applies.
\end{proof}
Hurwitz curves can also be described as follows: they are precisely the complex curves that can be written as $\Gamma\backslash\mathfrak{H}$ where $\mathfrak{H}$ is the upper half plane and $\Gamma$ a normal subgroup of finite index in the triangle group $\Delta$ of signature $(2,3,7)$. Since the latter is finitely generated, there are only finitely many such normal subgroups of a given index. As the index is related to the genus of the corresponding curve by the formula
$$2-2g(\Gamma\backslash\mathfrak{H})=\chi (\Gamma\backslash\mathfrak{H} )=-\frac{(\Delta : \Gamma )}{42},$$
we see that there are only finitely many Hurwitz curves up to isomorphism in each genus. Let $\mathcal{H}_g$ be the finite set of isomorphism classes of Hurwitz curves in genus $g$.
\begin{Corollary}
The function $h\colon\bbn\to\bbn$ which maps $g\in\bbn$ to the cardinality of $\mathcal{H}_g$ is unbounded.
\end{Corollary}
\begin{proof}
Assume it were bounded, say $\lvert\mathcal{H}_g\rvert\le N$ for all $g$. Let $G_g$ be the finite quotient of $\Gal (\qbar /\bbq )$ that operates faithfully on $\mathcal{H}_g$; by assumption, $G_g$ is isomorphic to a subgroup of the symmetric group $\mathfrak{S}_N$ and hence has order dividing $N!$.

Let $\sigma\in\Gal (\qbar /\bbq )$ be of infinite order. Then $\sigma^{N!}$ is still of infinite order, but it projects to the identity in all $G_g$, hence operates trivially on all Hurwitz curves. Contradiction!
\end{proof}
These observations should be compared with the relative rarity of Hurwitz curves. For instance, the Dirichlet series
$$\sum_{g=1}^{\infty }\frac{h(g)}{g^s}=\sum_{\substack{\text{$X$ Hurwitz curve}\\ \text{up to isomorphism}}} \frac{1}{g(X)^s}$$
converges precisely for $\Re (s)>\frac{1}{3}$, see \cite{MR1882031}, which may be intepreted as saying that the genus of the $n$-th Hurwitz curve is of the order of magnitude $n^3$, in a very rough sense. In particular, $h(g)=0$ for almost all $g$.

Furthermore, Conder computed that there are only $92$ Hurwitz curves of genus less than one million, with only $32$ different genera occurring, see~\cite{MR887205}. From \cite{MR887205} we read that the only $g\le 100$ occurring are 3, 7, 14 and 17, and the tables in \cite{MR3010126} tell us about their behaviour under $\Gal (\qbar /\bbq )$:
\begin{enumerate}
\item The only Hurwitz curve in genus three is \emph{Klein's quartic curve} with homogeneous equation $x^3y+y^3z+z^3x=0$, hence fixed by $\Gal (\qbar /\bbq )$.
\item The only Hurwitz curve in genus seven is the \emph{Fricke--Macbeath curve}, discovered by Fricke~\cite{MR1511059} and rediscovered by Macbeath~\cite{MR0177342}. With homogeneous coordinates $x_1$, $x_2$, $x_3$, $y$, $z$ on $\bbp^4$ it is the curve $X\subset\bbp^4$ defined by the equations
$$x_j^2z^2=(y-\zeta_7^jz)(y-\zeta_7^{j+2}z)(y-\zeta_7^{j+3}z)(y-\zeta_7^{j+4}z)$$
for $j=1,2,3$, where $\zeta_7$ is a primitive seventh root of unity. Since it is the only Hurwitz curve in genus seven, it must be fixed by the entire group $\Gal (\qbar /\bbq )$. Indeed, in \cite{Hidalgo2012} an extremely complicated model over $\bbq$ was found.
\item In genus fourteen there are three Hurwitz curves known as the \emph{first Hurwitz triplet}. They are defined over $k=\bbq (\cos\frac{2\pi}{7})$ and permuted simply transitively by $\Gal (k/\bbq )$.
\item Finally, in genus seventeen there are two Hurwitz curves, defined over $\bbq (\sqrt{-3})$ and exchanged by this field's nontrivial automorphism.
\end{enumerate}

\section{Congruence and noncongruence Hurwitz curves}

\noindent Yet another conclusion of Corollary~\ref{TheoremA} is that noncongruence subgroups of the triangle group $\Delta$ are abundant. Recall that $\Delta$ is in fact an arithmetic Fuchsian group, which can be described as follows. Let $k=\bbq (\cos\frac{2\pi }{7})$, considered as a subfield of $\bbr$, and let $B$ be the quaternion algebra over $k$ which is unramified over each finite prime of $k$ and the identity embedding $k\to\bbr$ and ramified over the two non-identity embeddings $k\to\bbr$. The identity embedding $k\to\bbr$ then extends to an embedding $B\hookrightarrow\mathrm{M}_2(\bbr )$, unique up to $\GL_2(\bbr )$-conjugation. There is a unique maximal order $\mathcal{O}\subset B$ up to $B^{\times }$-conjugation, and under the embedding $B\hookrightarrow\mathrm{M}_2(\bbr )$ the group $\mathcal{O}^1$ of units of norm one in $\mathcal{O}$ is mapped to a lattice $\mathcal{O}^1\subset\SL_2(\bbr )$. Its image in $\PSL_2(\bbr )$ is then precisely $\Delta$.

Let us say that a Hurwitz curve $X$ is a \emph{congruence Hurwitz curve} if the corresponding subgroup of $\Delta$ is a congruence subgroup. Examples of congruence Hurwitz curves are the \emph{Hurwitz--Macbeath curves} obtained from principal congruence subgroups $\Delta (\mathfrak{p})$ for prime ideals $\mathfrak{p}$ of $k$, see \cite{MR2306150, MR0262379}. Most of the small genus examples given above are of this type: (i) corresponds to $\mathfrak{p}\mid 7$, (ii) to $\mathfrak{p}=2$, (iii) to $\mathfrak{p}\mid 13$. The groups corresponding to (iv), however, are noncongruence subgroups. They are contained in $\Delta (\mathfrak{p})$ with $\mathfrak{p}\mid 7$, with the intermediate quotient isomorphic to $(\bbz /2\bbz )^3$. It is easy to see that a group extension of $\Delta /\Delta (\mathfrak{p})\simeq\PSL (2,\bbz /7\bbz )$ by $(\bbz /2\bbz )^3$ cannot be a quotient of $\PSL (2,\mathfrak{o}_k/\mathfrak{n})$ for any ideal $\mathfrak{n}$ of $\mathfrak{o}_k$, as it would have to be (by \cite{MR2306150}) if the subgroup we are considering were a congruence subgroup.

From Shimura's theory of canonical models for quotients of the upper half plane by congruence groups \cite{MR0204426} we learn:
\begin{Proposition}\label{ActionOnCongruenceRatherSmall}
The action of $\Gal (\qbar /\bbq )$ on congruence Hurwitz curves factors through $\Gal (k^{\mathrm{ab}}/\bbq )$, where $k^{\mathrm{ab}}$ is the maximal abelian extension of $k=\bbq (\cos \frac{2\pi }{7})$. In particular it is not faithful.
\end{Proposition}
\begin{proof}
Every congruence Hurwitz curve has a canonical model over a certain class field of $k$.
\end{proof}
The author is skeptical whether the induced action of $\Gal (k^{\mathrm{ab}} /\bbq )$ is faithful; he suspects that it is not, and that the action may even factor through $\Gal (k/\bbq )$, since the field of definition of the canonical model is usually much larger than the field of moduli. For instance, combining \cite{MR2306150} and \cite{MR0262379} one sees that $\Gal (\qbar /\bbq )$ acts on the Hurwitz curves corresponding to the subgroups $\Delta (\mathfrak{p})$ with prime $\mathfrak{p}$ by permuting the primes in the obvious way, in particular the field of moduli of such a curve is the splitting field in $k/\bbq$ of the rational prime below $\mathfrak{p}$.

\section{Profinite completion and Galois action on torsion}

\noindent The Galois action on Hurwitz curves can also be described in a more group-theoretical manner. Let $\hat{\Delta }$ be the profinite completion of $\Delta$. Then open normal subgroups of $\hat{\Delta }$ correspond to normal subgroups of finite index in $\Delta$ and hence to isomorphism classes of Hurwitz curves. There is a continuous exterior action of $\Gal (\qbar /\bbq )$ on $\hat{\Delta }$, i.e.\ a continuous homomorphism
\begin{equation*}
\varrho_{\Delta }\colon \Gal (\qbar /\bbq )\to\Out \hat{\Delta }\defined\Aut\hat{\Delta }/\Inn\hat{\Delta }
\end{equation*}
where $\Inn\hat{\Delta }$ denotes the group of inner automorphisms. This action is constructed explicitly and described in~\cite{GonzalezDiez2013} for general triangle groups. The Galois action on open normal subgroups of $\hat{\Delta }$ it induces is precisely the Galois action on Hurwitz curves discussed before.

One way to construct $\varrho_{\Delta }$ is by a stack version of Grothendieck's theory of the fundamental group: There exists a Deligne--Mumford stack $\mathcal{X}$ over $\bbq$ whose coarse moduli space is $\bbp^1$ and which has trivial generic stabilisers, and cyclic stabilisers of orders 2, 3 and 7 at $0$, $1$ and $\infty$, respectively. The analytification of $\mathcal{X}_{\bbc }$ is then the orbifold quotient $[\Delta\backslash\mathfrak{H}]$, and hence we can identify the \'{e}tale fundamental group of $\mathcal{X}_{\qbar }$ with $\hat{\Delta }$. The usual short exact ``homotopy'' sequence for \'{e}tale fundamental groups \cite[IX.6.1]{MR0354651} exists also in this case:
\begin{equation}
1\to\pioneet (\mathcal{X}_{\qbar },\ast )\to\pioneet (\mathcal{X},\ast )\to\Gal (\qbar /\bbq )\to 1.
\end{equation}
The homomorphism $\Gal (\qbar /\bbq )\to\Out\pioneet (\mathcal{X}_{\qbar },\ast )$ defined by this sequence is $\varrho_{\Delta }$.

Theorem~\ref{GDJZ} clearly implies that $\varrho_{\Delta }$ is injective. The proof of Theorem~\ref{GDJZ} in~\cite{GonzalezDiez2013}, however, goes in the opposite direction and uses this very injectivity (for general triangle groups), which follows from results of Hoshi and Mochizuki~\cite{MR2895284} in anabelian geometry.

For the next observation, we use that every Hurwitz curve admits a model over its field of moduli, which follows from~\cite{MR0453746}.

\begin{Corollary}\label{TheoremCprime}
Fix an element $\sigma\in \Gal (\qbar /\bbq )$ other than the identity. Then there exists a Hurwitz curve $Y$ with moduli field $\bbq$ such that for any model $Y_0$ of $Y$ over $\bbq$ and for every odd prime $\ell$, the image of $\sigma$ under the representation
$$\varrho_{Y_0,\ell}\colon \Gal(\qbar /\bbq )\to\GL (2g,\bbf_{\ell })$$
is not the identity.
\end{Corollary}
Here $\varrho_{Y_0,\ell}$ is the usual Galois representation on the $\ell$-torsion of the Jacobian, $(\Jac Y)[\ell ]\cong\bbf_{\ell }^{2g}$.

\begin{proof}[Proof of Corollary~\ref{TheoremCprime}]
Every open normal subgroup $N$ of $\hat{\Delta }$ contains one which is stable under $\Gal (\qbar /\bbq )$: the setwise stabiliser of $N$ in $\Gal (\qbar /\bbq )$ has finite index in $\Gal (\qbar /\bbq )$, therefore
$$\tilde{N}=\bigcap_{\sigma\in \Gal (\qbar /\bbq )}\sigma (N)$$
is an open normal subgroup of $\hat{\Delta }$ contained in $N$. This means that $\hat{\Delta }$ can also be described as the projective limit of all $\hat{\Delta }/N$ with $N$ open, normal and stable under $\Gal (\qbar /\bbq )$. We conclude (using the compactness of $\hat{\Delta }$) that $\sigma$ operates by a non-trivial outer automorphism on some such $\hat{\Delta }/N$. Now $N$ corresponds to a Hurwitz curve $Y$ with moduli field $\bbq$; we claim that $Y$ has the desired properties.

The Hurwitz group $H=\hat{\Delta }/N=\Aut_{\qbar }Y$ sits in a short exact sequence:
\begin{equation}
1\to\Aut_{\qbar }Y\to\Aut_{\bbq }Y\to\Gal (\qbar /\bbq )\to 1.
\end{equation}
Here the middle term means the group of all automorphisms of $Y$ as a $\bbq$-scheme (or, which amounts to the same, as a scheme without any further structure). A choice of a model $Y_0$ over $\bbq$ yields a splitting $s$ of this sequence.

Now $\Aut_{\bbq }Y$ acts naturally on the \'{e}tale cohomology group $\mathrm{H}^1(Y,\bbf_{\ell })$; by~\cite{Serre60} the subgroup $H=\Aut_{\qbar }Y$ operates faithfully on this cohomology group. There exists some $h\in H$ such that $s(\sigma )hs(\sigma )^{-1}\neq h$, for otherwise $\sigma$ would act trivially, in particular as an inner automorphism, on $H$. Hence these two elements also operate differently on $\mathrm{H}^1(Y,\bbf_{\ell })$. But this means that $s(\sigma )$ has to operate nontrivially on this cohomology group. Finally, the $\ell$-torsion points of the Jacobian are canonically identified with the dual of $\mathrm{H}^1(Y,\bbf_{\ell })$, so $\sigma$ also operates nontrivially there.
\end{proof}

\section{Hurwitz origamis}

\noindent Instead of searching for algebraic curves with large automorphism group compared to the genus one might do the same for translation surfaces. A translation surface is a closed Riemann surface with a nonzero holomorphic one-form; for more geometric descriptions, see~\cite{MR2186246}. Such a search has been initiated in~\cite{SchmithuesenSchlagePuchta}. There it is shown that a translation surface of genus $g\ge 2$ has at most $4(g-1)$ automorphisms, and surfaces achieving this bound are named Hurwitz translation surfaces. Since they belong to a particularly nice class of translation surfaces called origamis, we suggest to use the more colourful term \emph{Hurwitz origamis} instead. They are more common than Hurwitz curves; for example, a Hurwitz origami exists in genus $g$ if and only if $g\equiv 1,3,4,5\bmod 6$, see \cite[Theorem~2]{SchmithuesenSchlagePuchta}.

An origami is a finite ramified covering of a square torus $\bbc /\bbz [\mathrm{i}]$, unramified outside the origin. Four the arithmetic point of view taken here is it sensible to rephrase this: an origami is a smooth projective complex curve $X$ together with a nonconstant rational map $p\colon X\to E$, where $E$ is the complex elliptic curve with $j(E)=1728$, which is unramified outside $0\in E(\bbc )$. Just as for Bely\u{\i} pairs, the curve $X$ and the map $p$ descend uniquely to objects over $\qbar$. We then have a Galois action on pairs $(X,p)$ as soon as we fix a model of $E$ over $\bbq$ (for instance, $y^2=x^4-1$). The actions corresponding to two different models will still agree on an open normal subgroup of $\Gal (\qbar /\bbq )$.

Hurwitz origamis are then precisely those pairs $(X,p)$ where $p$ is a normal covering that has ramification order two at all ramification points --- this is a reformulation of \cite[Theorem~1]{SchmithuesenSchlagePuchta}. In particular, being a Hurwitz origami is stable under the Galois action. The following observation then holds regardless of the model of $E$ chosen:
\begin{Theorem}\label{TheoremC}
The absolute Galois group operates faithfully on the set of isomorphism classes of Hurwitz origamis.
\end{Theorem}
\begin{proof}[Sketch of proof]
There exists a Deligne--Mumford stack $\mathcal{E}$ over $\bbq$ whose coarse moduli space is the chosen model of $E$ over $\bbq$, which stabiliser $\bbz /2\bbz$ at the origin and trivial stabilisers otherwise. The analytification of $\mathcal{E}_{\bbc }$ is an orbifold quotient $[Q\backslash\mathfrak{H}]$ where $Q\subset\SL_2(\bbr )$ is a quadrilateral group. The profinite completion $\hat{Q}$ can then be identified with the \'{e}tale fundamental group of $\mathcal{E}_{\qbar }$, and just as for Hurwitz curves we obtain a short exact sequence
$$1\to\pioneet (\mathcal{E}_{\qbar },\ast )\to\pioneet (\mathcal{E},\ast )\to\Gal (\qbar /\bbq )\to 1$$
and a continuous homomorphism $\varrho_Q\colon\Gal (\qbar /\bbq )\to\Out\hat{Q}$. It follows then again from \cite[Theorem~C]{MR2895284} that $\varrho_Q$ is injective.

Next we use a slight generalisation of \cite[Theorem~27]{GonzalezDiez2013}, which is formulated only for triangle groups but holds more generally with a similar proof for arbitrary cocompact Fuchsian groups. The case we need, which can be proved completely analogously to the special case discussed in \cite{GonzalezDiez2013}, states that a continuous group automorphism of $\hat{Q}$ which preserves all open normal subgroups must already be an inner automorphism.

Now we assume that $\sigma\in\Gal (\qbar /\bbq )$ preserves all Hurwitz origamis, that is, $\varrho_Q(\sigma )$ preserves all open normal subgroups of $\hat{Q}$. Hence it is trivial in $\Out\hat{Q}$, therefore $\sigma$ has to be the identity.
\end{proof}
The previous sketch of proof is extended to a detailed proof in \cite{MyThesis}.

Finally we remark that a statement analogous to Corollary~\ref{TheoremCprime} also holds for Hurwitz origamis, as long as ``with moduli field $\bbq$'' is replaced with ``admitting a model over $\bbq$''.

\providecommand{\bysame}{\leavevmode\hbox to3em{\hrulefill}\thinspace}
\providecommand{\MR}{\relax\ifhmode\unskip\space\fi MR }
\providecommand{\MRhref}[2]{%
  \href{http://www.ams.org/mathscinet-getitem?mr=#1}{#2}
}
\providecommand{\href}[2]{#2}


\begin{thebibliography}{10}

\bibitem{MR0354651}
\emph{Rev\^etements \'etales et groupe fondamental}, Springer-Verlag, Berlin,
  1971, S{\'e}minaire de G{\'e}om{\'e}trie Alg{\'e}brique du Bois Marie
  1960--1961 (SGA 1), Dirig{\'e} par Alexandre Grothendieck. Augment{\'e} de
  deux expos{\'e}s de M. Raynaud, Lecture Notes in Mathematics, Vol. 224.
  \MR{0354651 (50 \#7129)}

\bibitem{MR534593}
G.~V. Bely{\u\i}, \emph{Galois extensions of a maximal cyclotomic field}, Izv.
  Akad. Nauk SSSR Ser. Mat. \textbf{43} (1979), no.~2, 267--276, 479.
  \MR{534593 (80f:12008)}

\bibitem{BogomolovHusemoeller2000}
F.~Bogomolov and D.~Husem\"{o}ller, \emph{Geometric properties of curves
  defined over number fields}, preprint, Max-Planck-Institut f\"{u}r
  Mathematik, Bonn.

\bibitem{MR887205}
Marston Conder, \emph{The genus of compact {R}iemann surfaces with maximal
  automorphism group}, J. Algebra \textbf{108} (1987), no.~1, 204--247.
  \MR{887205 (88f:20063)}

\bibitem{MR3010126}
Marston D.~E. Conder, Gareth~A. Jones, Manfred Streit, and J{\"u}rgen Wolfart,
  \emph{Galois actions on regular dessins of small genera}, Rev. Mat. Iberoam.
  \textbf{29} (2013), no.~1, 163--181. \MR{3010126}

\bibitem{MR2306150}
Amir D{\v{z}}ambi{\'c}, \emph{Macbeath's infinite series of {H}urwitz groups},
  Arithmetic and geometry around hypergeometric functions, Progr. Math., vol.
  260, Birkh\"auser, Basel, 2007, pp.~101--108. \MR{2306150 (2008b:20062)}

\bibitem{MR1511059}
Robert Fricke, \emph{Ueber eine einfache {G}ruppe von 504 {O}perationen}, Math.
  Ann. \textbf{52} (1899), no.~2-3, 321--339. \MR{1511059}

\bibitem{MR0453746}
M.~Fried, \emph{Fields of definition of function fields and {H}urwitz
  families---groups as {G}alois groups}, Comm. Algebra \textbf{5} (1977),
  no.~1, 17--82. \MR{0453746 (56 \#12006)}

\bibitem{MR2895884}
Ernesto Girondo and Gabino Gonz{\'a}lez-Diez, \emph{Introduction to compact
  {R}iemann surfaces and dessins d'enfants}, London Mathematical Society
  Student Texts, vol.~79, Cambridge University Press, Cambridge, 2012.
  \MR{2895884}

\bibitem{GonzalezDiez2013}
Gabino Gonz\'alez-Diez and Andrei Jaikin-Zapirain, \emph{The absolute {G}alois
  group acts faithfully on regular dessins and on {B}eauville surfaces},
  preprint,
  http://www.uam.es/personal{\underline{~}}pdi/ciencias/gabino/VersionWeb.pdf,
  2014.

\bibitem{MR1483107}
Alexandre Grothendieck, \emph{Esquisse d'un programme}, Geometric {G}alois
  actions, 1, London Math. Soc. Lecture Note Ser., vol. 242, Cambridge Univ.
  Press, Cambridge, 1997, With an English translation on pp. 243--283,
  pp.~5--48. \MR{1483107 (99c:14034)}

\bibitem{MR2349672}
William~J. Harvey, \emph{Teichm\"uller spaces, triangle groups and
  {G}rothendieck dessins}, Handbook of {T}eichm\"uller theory. {V}ol. {I}, IRMA
  Lect. Math. Theor. Phys., vol.~11, Eur. Math. Soc., Z\"urich, 2007,
  pp.~249--292. \MR{2349672 (2009h:32019)}

\bibitem{MR2516744}
Frank Herrlich and Gabriela Schmith{\"u}sen, \emph{Dessins d'enfants and
  origami curves}, Handbook of {T}eichm\"uller theory. {V}ol. {II}, IRMA Lect.
  Math. Theor. Phys., vol.~13, Eur. Math. Soc., Z\"urich, 2009, pp.~767--809.
  \MR{2516744 (2010f:14036)}

\bibitem{Hidalgo2012}
Ruben~A. Hidalgo, \emph{A computational note about {F}ricke--{M}acbeath's
  curve}, preprint, arXiv:1203.6314, 2012.

\bibitem{MR2895284}
Yuichiro Hoshi and Shinichi Mochizuki, \emph{On the combinatorial anabelian
  geometry of nodally nondegenerate outer representations}, Hiroshima Math. J.
  \textbf{41} (2011), no.~3, 275--342. \MR{2895284}

\bibitem{MR2186246}
Pascal Hubert and Thomas~A. Schmidt, \emph{An introduction to {V}eech
  surfaces}, Handbook of dynamical systems. {V}ol. 1{B}, Elsevier B. V.,
  Amsterdam, 2006, pp.~501--526. \MR{2186246 (2006i:37099)}

\bibitem{MR1510753}
A.~Hurwitz, \emph{Ueber algebraische {G}ebilde mit eindeutigen
  {T}ransformationen in sich}, Math. Ann. \textbf{41} (1892), no.~3, 403--442.
  \MR{1510753}

\bibitem{MR2070647}
Bernhard K{\"o}ck, \emph{Belyi's theorem revisited}, Beitr\"age Algebra Geom.
  \textbf{45} (2004), no.~1, 253--265. \MR{2070647 (2005j:14036)}

\bibitem{MyThesis}
Robert~A. Kucharczyk, \emph{On arithmetic properties of {F}uchsian groups and
  {R}iemann surfaces}, 2015, PhD thesis, Rheinische
  Friedrich-Wilhelms-Universit\"{a}t Bonn.

\bibitem{MR1882031}
Michael Larsen, \emph{How often is {$84(g-1)$} achieved?}, Israel J. Math.
  \textbf{126} (2001), 1--16. \MR{1882031 (2002m:30056)}

\bibitem{MR0177342}
A.~M. Macbeath, \emph{On a curve of genus {$7$}}, Proc. London Math. Soc. (3)
  \textbf{15} (1965), 527--542. \MR{0177342 (31 \#1605)}

\bibitem{MR0262379}
\bysame, \emph{Generators of the linear fractional groups}, Number {T}heory
  ({P}roc. {S}ympos. {P}ure {M}ath., {V}ol. {XII}, {H}ouston, {T}ex., 1967),
  Amer. Math. Soc., Providence, R.I., 1969, pp.~14--32. \MR{0262379 (41
  \#6987)}

\bibitem{MR1722414}
\bysame, \emph{Hurwitz groups and surfaces}, The eightfold way,
  Math. Sci. Res. Inst. Publ., vol.~35, Cambridge Univ. Press, Cambridge, 1999,
  pp.~103--113. \MR{1722414 (2001c:14002)}

\bibitem{MR2074061}
Joseph Oesterl{\'e}, \emph{Dessins d'enfants}, Ast\'erisque (2003), no.~290,
  Exp. No. 907, ix, 285--305, S{\'e}minaire Bourbaki. Vol. 2001/2002.
  \MR{2074061 (2006c:14031)}

\bibitem{SchmithuesenSchlagePuchta}
Jan-Christoph Schlage-Puchta and Gabriela Weitze-Schmith{\"u}sen, \emph{Finite
  translation surfaces with maximal number of translations}, preprint,
  arXiv:1311.7446v1, 2013.

\bibitem{MR1305393}
Leila Schneps, \emph{Dessins d'enfants on the {R}iemann sphere}, The
  {G}rothendieck theory of dessins d'enfants ({L}uminy, 1993), London Math.
  Soc. Lecture Note Ser., vol. 200, Cambridge Univ. Press, Cambridge, 1994,
  pp.~47--77. \MR{1305393 (95j:11061)}

\bibitem{Serre60}
Jean-Pierre Serre, \emph{Rigidit{\'e} du foncteur de {J}acobi d'{\'e}chelon
  $n\ge 3$}, S\'eminaire Henri Cartan 1960/61, Appendice \`a l'Expos\'e 17,
  Secr\'etariat math\'ematique, Paris 1961.

\bibitem{MR0204426}
Goro Shimura, \emph{Construction of class fields and zeta functions of
  algebraic curves}, Ann. of Math. (2) \textbf{85} (1967), 58--159. \MR{0204426
  (34 \#4268)}

\end{thebibliography}
\end{document}